\documentclass[leqno,a4paper,12pt]{amsart}
\usepackage{amsmath,amsthm,amssymb}
\usepackage[utf8]{inputenc}
\usepackage[T1]{fontenc}

\usepackage{enumerate}
\usepackage{bbm}
\usepackage{todonotes}
\usepackage{mathrsfs}
\usepackage{mathabx}
\usepackage{hyperref}
\usepackage{nicefrac}
\usepackage{tikz}
\usetikzlibrary{matrix}

\usepackage{url}
\newtheorem{Th}{Theorem}[section]

\newtheorem{Lemma}[Th]{Lemma}
\newtheorem{Cor}[Th]{Corollary}

\theoremstyle{definition}

\newtheorem{Def}{Definition}[section]
\newtheorem{Example}{Example}[section]

\newtheorem{Fact}[Th]{Fact}

\newcommand{\beq}{\begin{equation}}
\newcommand{\eeq}{\end{equation}}
\newcommand{\abs}[1]{\left\lvert #1\right\rvert}

\def\scalar(#1,#2){(#1\mid#2)}

\DeclareMathOperator{\ap}{ap}
\DeclareMathOperator{\per}{per}
\DeclareMathOperator{\Per}{Per}
\DeclareMathOperator{\lcm}{lcm}

\newcommand{\Q}{\mathbb{Q}}

\newcommand{\C}{{\mathbb{C}}}
\newcommand{\Z}{{\mathbb{Z}}}
\newcommand{\N}{{\mathbb{N}}}

\begin{document}
\title{Distribution of polynomial orbits in Toeplitz systems}
\author{Kosma Kasprzak}
\maketitle
\begin{abstract}
We examine the convergence of ergodic averages along polynomials in Toeplitz systems and prove that it is possible for averages along one polynomial to converge, and along another to diverge. We also study density of the polynomial orbits in Toeplitz systems -- we show that it implies equidistribution of the polynomial orbits in the class of regular Toeplitz systems, but not in the class of strictly ergodic ones.
\end{abstract}
\maketitle

\section{Introduction}

For a topological dynamical system $(X, T)$ and a sequence of integers $(a_n)$ we can consider the sequence of iterates $(T^{a_n}y)$ for a point $y\in X$, which we call the orbit along $(a_n)$. The question of distribution of such sparse orbits enjoys a lot of recent activity, particularly for the sequence of primes or sequences $(P(n))$ of values of polynomials $P\in\Z[X]$. If $(X, T)$ is uniquely ergodic, one naturally considers equidistribution of the sparse orbits, i.e. the property that for any $y\in X$ and $F\in C(X)$ the averages
\begin{equation}\label{avgconv}
\frac{1}{N}\sum\limits_{n<N}F(T^{a_n}y)
\end{equation} converge to $\int_X F d\mu$, where $\mu$ is the unique $T$-invariant measure on $X$. This is however sometimes too restrictive because of local obstructions (i.e. because $(a_n)$ might not be well distributed in residue classes modulo integers), and the more interesting question is whether the sequence of averages (\ref{avgconv}) is convergent for every $y\in X$. Recall that by the work of Bourgain and Wierdl, this is true for almost every $y\in X$ along primes or polynomials.

Our goal will be to examine these properties, as well as density of the sparse orbits, in the class of Toeplitz systems, mainly along polynomial sequences $(P(n))$ for $P\in\Z[X]$. 

In \cite{Toeplitz}, the authors consider the distribution of orbits along primes in Toeplitz systems, focusing on the so-called regular Toeplitz systems, which are in particular strictly ergodic. The authors also treat polynomial orbits, and find an example of a Toeplitz system $X_x$ where averages (\ref{avgconv}) converge for the sequence of primes for every starting point, but not for $(n^2)$ with $y=x$.

We prove a similar result distinguishing orbits along two different polynomials:
\begin{Th}\label{mainmain}
Let $k, l>1$ be distinct integers. Then there exists a regular Toeplitz system $X_x$ over the alphabet $\{0, 1\}$ such that for every continuous map $F:X_x\to\C$ and $y\in X_x$,  the limit
\[
\lim_{N\to\infty}\frac{1}{N}\sum_{m\leq N}F(\sigma^{m^l}y)
\]
exists, but the limit
\[
\lim_{N\to\infty}\frac{1}{N}\sum_{m\leq N}G(\sigma^{m^k}x)
\]
does not exist for the continuous function $G(y)=(-1)^{y(0)}$.

Moreover, if $l$ is odd, and $q|k$ and $q\nmid l$ for some prime $q$, we can ensure that the orbit $(\sigma^{m^l}y)$ of any point $y\in X_x$ is equidistributed in $X_x$ with respect to the unique $\sigma$-invariant measure on $X_x$.  
\end{Th}
This is the first known example of a minimal topological dynamical system $(X, T)$ where sequences of averages (\ref{avgconv}) converge for all $y\in X$ for $a_m=m^l$, but fail to converge for $a_m=m^k$, where $k, l>1$. 

The main difficulty in finding such an example is that there are only a few classes of systems where we have any control over averages along polynomials. Until recently, essentially the only known way of doing this relied on variants of the van der Corput inequality. This method can only be used in systems with some algebraic structure, such as nilsystems. Moreover, it does not seem to be able to distinguish one polynomial from another -- in all known cases where the method can be applied, it applies to all polynomials. 

A new approach relying on rigidity is developed in the article \cite{rigid}, resulting in particular in a class of weakly mixing systems where all square orbits are equidistributed. The method does not however give control over averages along polynomials of higher degree.

We are successful in considering Toeplitz systems in particular because they are quite simple systems -- one can think of them as limits of finite systems, where averages (\ref{avgconv}) can be very well understood by number-theoretic methods.\\

In Section \ref{main1} we construct the system in Theorem \ref{mainmain} in the case $k\nmid l$, adapting methods from \cite{Toeplitz}. One result needed to establish the last sentence of the theorem in this case is postponed to Section 5. 

The case $k|l$ is much more subtle -- we deal with it in Section \ref{main2}, using the Weil bound to estimate the number of solutions $(x, y)\in\mathbb{F}_p\times \mathbb{F}_p$ to equations of the form $x^k=y^l+a$.

In Section \ref{dense} we turn to discussing density of polynomial orbits in Toeplitz systems. Our Theorems \ref{dense} and \ref{resfull} essentially characterize all cases when such density occurs, in terms of the integers modulo which the polynomial is a permutation. In regular Toeplitz systems this characterization gives
\begin{Th}\label{denseequi}
Let $X_x$ be a regular Toeplitz system and let $P\in\Z[X]$. If the orbit $(\sigma^{P(n)}y)$ is dense in $X_x$ for some $y$, then the orbits $(\sigma^{P(n)}y)$ are equidistributed in $X_x$ for all $y\in X_x$.
\end{Th}
Hence any examples where the $(P(n))$-orbits are dense, but not equidistributed, cannot be regular. Let us mention that strict ergodicity of a Toeplitz system does not imply its regularity -- a suitable counterexample is given in \cite{iwanik}.

We adapt this example in Theorem \ref{iwacounter}, showing that regularity cannot be replaced by strict ergodicity in Theorem \ref{denseequi}.

\section{Preliminaries and facts}
\subsection{Toeplitz systems}\label{toep}
For simplicity we will only consider Toeplitz systems over the binary alphabet $\{0, 1\}$. We call $x\in \{0, 1\}^{\Z}$ a Toeplitz word if for every $a\in \Z$ there exists an $\ell\in\Z^+$ such that $x(a)=x({a+k\ell})$ for each $k\in\Z$. The orbit closure $X_x$ of $x$ in $\{0, 1\}^{\Z}$ is then a minimal topological dynamical system.
Every Toeplitz word $x$ admits a sequence $(n_t)$ satisfying $n_t|n_{t+1}$ such that if 
\[
{\rm Per}_{n_t}(x):=\{a\in\Z:\: x(a)=x(a+kn_t)\text{ for all }k\in\Z\},
\]

then \begin{equation}\label{max}\bigcup_{t\geq0}{\rm Per}_{n_t}(x)=\Z.\end{equation}
We call the system $X_x$ \textit{regular}, if the density of $\bigcup_{t=0}^M{\rm Per}_{n_t}(x)$ converges to $1$. This condition implies strict ergodicity (see Theorem 13.1 in \cite{Downarowicz}). \\

In \cite{Toeplitz} the authors present the following general way of constructing Toeplitz systems.
\begin{Def}
Let $(n_t)$ be a sequence of natural numbers satisfying $n_t|n_{t+1}$, and $x_t\in\{0, 1, ?\}^{n_t}$. We call the pair of sequences $(n_t), (x_t)$ \textit{viable} if 
\begin{itemize}
\item The word $x_{t+1}$ is a concatenation of $n_{t+1}/n_t$ copies of $x_t$, with possibly some question marks replaced with 0s or 1s.

(in other words, whenever $x_t(i)\ne ?$ for some $t\in\N$ and $i\in [0, n_t-1]$, then $x_s(j)=x_t(i)$ for all $s\geqslant t$ and $j\equiv i\pmod{n_t}$ satisfying $j\in [0, n_{s}-1]$)

\item for any $i\in\mathbb{N}$ there exists $t$ such that $x_t(i), x_t(n_t-i)\ne ?$
\end{itemize}

To such sequences $(n_t), (x_t)$ we associate the sequence
\[
?_t=\abs{\{i: x_t(i)=?\}\subset [0, n_t-1]\}}.
\]
\end{Def}

It is not difficult to see the following:
\begin{Fact}
Let $(n_t), (x_t)$ be viable. Then the word $x\in \{0, 1\}^\Z$ defined as
\[
x(i)=\begin{cases}x_t(i) \quad\text{for $i\geqslant 0$}\\ x_t(n_t+i) \quad\text{for $i<0$} \end{cases}
\]
is Toeplitz (here $t$ is the smallest index for which the relevant symbol is not "?"). All Toeplitz words can be obtained this way. Additionally, if $?_t=o(n_t)$, then $X_x$ is regular.
\end{Fact}

\begin{Def}
Fix an integer $k>1$.
For all $N,n, a\in\Z$, let
\[\rho_k(N;n,a)=\abs{\{1\leq m\leq N:m^k=a\!\mod n\}}.\]
and
\[\rho_k(n):=\max_{a\in \Z}\rho_k(n;n,a).\]
\end{Def}
We have
\begin{Th}[\cite{Toeplitz}, Theorem 6.8]\label{thm:posP}
Suppose that $X_x$ is a Toeplitz system such that
\[
?_t=o(n_t/{\rho_k(n_t)}).
\]
Then, for every continuous map $F:X_x\to\C$ and $y\in X_x$,  the limit
\[
\lim_{N\to\infty}\frac{1}{N}\sum_{m\leq N}F(\sigma^{m^k}y)
\]
exists.
\end{Th}
We will also need the following
\begin{Def}
We denote
\[R^k_n=\{0\leqslant a<n: a\equiv m^k\pmod{n} \text{ for some $m\in\Z$}\}\]
and

\[\widetilde{R^k_n}=\{a\in R^k_n: \gcd(a, n)=1\}.\]
\end{Def}

We now briefly describe the method used in \cite{iwanik} to construct a non-regular strictly ergodic Toeplitz system. As opposed to the previous construction, it can detect strict ergodicity. We will use it in the proof of Theorem \ref{iwacounter}.

Let $(n_t)$ be a sequence of positive integers such that $n_t|n_{t+1}$, and let $W_t\subset \{0, 1\}^{n_t}$ satisfy the following:
\begin{enumerate}
\item each word in $W_{t+1}$ is a concatenation of $n_{t+1}/n_t$ words from $W_t$
\item the word $\omega[0, n_t-1]\in W_t$ is the same for all $\omega\in W_{t+1}$
\item the word $\omega[n_{t+1}-n_t, n_{t+1}-1]\in W_t$ is the same for all $\omega\in W_{t+1}$
\end{enumerate}
Then we can pick any $\omega_t\in W_t$ and let $x\in\{0, 1\}^{\Z}$ be given by
\[
x(i)=\begin{cases}\omega_t(i)\quad &\text{for $i\geqslant 0$}\\ \omega_t(n_t+i)\quad &\text{for $i<0$}\end{cases},
\]
where $t$ satisfies $|i|<n_t$.
Notice that by conditions (2) and (3) the sequences $(\omega_t(i))$ and $(\omega_t(n_t+i))$ respectively are constant wherever they are defined.

It is not difficult to see the following:
\begin{Fact}\label{iwatoep}
The word $x$ constructed above is Toeplitz. Any Toeplitz word can be obtained this way.
\end{Fact}

Let $B\in\{0, 1\}^\beta$ and $C\in\{0, 1\}^{\gamma}$, where $\beta|\gamma$. The word $C$ can be split into $\gamma/\beta$ words of length $\beta$; by $\ap(B, C)$ we denote the frequency of appearances of $B$ among these words. To be precise,
\[
\ap(B, C)=\frac{\beta}{\gamma}\abs{\{i\in [0, \gamma/\beta-1]: C[i\beta, (i+1)\beta -1]=B\}}.
\]
\begin{Th}[\cite{iwanik}, Theorem 1.1]\label{unierg}
In the above notation, the system $X_x$ is strictly ergodic if and only if for any $s$ and any $B\in W_s$ there exists a number $\nu(B)$ such that
\[
\ap(B, C)\to \nu(B)
\]
uniformly in $C\in W_t$ as $t\to\infty$.
\end{Th}

\subsection{Permutative polynomials}

Due to the conclusion of Corollary (\ref{fullset}), we will consider the following:
\begin{Def}
We say that a polynomial $P\in \Z[X]$ is a \textit{permutation modulo $n$} if $P:\Z/n\Z\to\Z/n\Z$ is a bijection. We say $P$ is \textit{permutative} if it is a permutation modulo $n$ for infinitely many $n\in\Z^+$.
\end{Def}
\begin{Example}
For $k\in\Z^+$ the polynomial $P(x)=x^k$ is permutative if and only if $2\nmid k$.
\end{Example}

By the Chinese remainder theorem, permutativity of $P$ is equivalent to the statement that either $P$ is bijective modulo infinitely many primes, or that for some prime $p$ the polynomial $P$ is bijective modulo all $p^k$ for $k\in\Z^+$.
The second condition can be stated in a more concrete way: (for a proof see e.g. Corollary 4.3 in \cite{dickson})
\begin{Th}
Let $P\in\Z[X]$ and let $p$ be prime. The following conditions are equivalent:
\begin{enumerate}
\item $P$ is a permutation modulo $p^k$ for every $k\in\Z^+$.
\item $P$ is a permutation modulo $p^2$
\item $P$ is a permutation modulo $p$, and $p\nmid P'(x)$ for all $x\in \Z$.
\end{enumerate}
\end{Th}
A description of polynomials which are permutations modulo infinitely many primes is much more difficult. It relies on the following object:
\begin{Def}
For $n\in\Z^+$ and $\alpha\in\Z$ we let the \textit{Dickins polynomial} $D_n(\alpha, x)$ be the unique polynomial in $\Z[X]$ of degree $n$ satisfying 
\[
D_n\left(\alpha, x+\frac{\alpha}{x}\right)=x^n+\left(\frac{\alpha}{x}\right)^n
\]
for all $x\in\Z$.
\end{Def}

\begin{Th}[Theorem 4, \cite{turnwald}]
Let $P\in\Z[X]$ be of degree $n$. Then the following conditions are equivalent:
\begin{enumerate}
\item $P$ is a permutation modulo infinitely many primes
\item $P$ is a composition of finitely many degree 1 polynomials in $\Q[X]$ and Dickins polynomials $D_{n_i}(\alpha_i, x)$ for (not neccessarily distinct) odd primes $n_i$ and numbers $\alpha_i\in\Z$, where whenever $3|n_i$ we have $\alpha_i=0$.
\end{enumerate}
\end{Th}
An analogue of the above result in the more general setting of rings of integers of number fields, proven in the cited article, is known as Schur's conjecture.

\section{Main result for $k\nmid l$}\label{main1}

\begin{Th}\label{polyequi}
Let $k, l\in \N$ be such that $k\nmid l$. Then there exists a regular Toeplitz system $X_x$ such that for every continuous map $F:X_x\to\C$ and $y\in X_x$,  the limit
\begin{equation}\label{eq:SNT2}
\lim_{N\to\infty}\frac{1}{N}\sum_{m\leq N}F(\sigma^{m^l}y)
\end{equation}
exists, but the limit
\begin{equation}\label{div}
\lim_{N\to\infty}\frac{1}{N}\sum_{m\leq N}G(\sigma^{m^k}x)
\end{equation}
does not exist for the continuous function $G(y)=(-1)^{y(0)}$.
Moreover, if $l$ is odd, and $q|k$ and $q\nmid l$ for some prime $q$, we can ensure that the orbit $(\sigma^{m^l}y)$ of any point $y\in X_x$ is equidistributed in $X_x$ with respect to the unique $\sigma$-invariant measure on $X_x$.  
\end{Th}
\begin{proof}
First we replace $k$ and $l$ by two smaller numbers. Below by $\nu_p(n)$ for prime $p$ and $n\in\Z^+$ we mean the largest integer satisfying $p^{\nu_p(n)}|n$.
\begin{Lemma}\label{ntlem}
There exist numbers $k'>l'\geqslant 1$ for which there are infinitely many primes $p$ satisfying
\[
\gcd(p-1, k)=k'\quad\text{and}\quad\gcd(p-1, l)=l'.
\]
If $l$ is odd, and $q|k$ and $q\nmid l$ for some prime $q$, then we can take $l'=1$.
\end{Lemma}
\begin{proof}
Since $k\nmid l$, there exists a prime $q$ for which $\nu_q(k)>\nu_q(l)$ (if the last sentence of the theorem statement is true, we can also assume $\nu_q(l)=1$). If $q=2$, then we can set $k'=2k'=2^{\nu_2(l)+1}$, and let $p$ be any prime satisfying
\[
\gcd(p-1, kl)=2^{\nu_2(l)+1}.
\]
Else we let $l'=q^{\nu_q(l)}\cdot \gcd(l, 2)$ and $k'=q^{\nu_q(k)+1}\cdot\gcd(k, 2)$, so that $k'/l'\geqslant q/2>1$. Then $p$ just has to satisfy
\[
\gcd(p-1, 2kl)=2q^{\nu_q(l)+1}.
\]
In both cases there are infinitely many possible primes $p$ by Dirichlet's theorem about primes in arithmetic progressions.
\end{proof}
We now recursively define the sequence $(n_t)$ by setting $n_0=1$ and letting $n_{t+1}=n_tp_{t+1}$ for a strictly increasing sequence $(p_t)$ of primes satisfying the condition in Lemma \ref{ntlem} and the following:
\begin{equation}\label{phicond} \frac{\varphi(n_t)}{n_t}>\frac{9}{10}\end{equation}
\begin{equation}\label{sumcond}\sum\limits_{t=1}^{\infty}\frac{8kn_t}{\sqrt{p_{t+1}}}<\frac{1}{10}\end{equation}
\begin{equation}\label{powcond}
p_{t+1}>30n_t^k.
\end{equation}
We also define a sequence $x_t\in \{0, 1, ?\}^{n_t}$. We start with $x_0=?$. Having defined $x_t$, we define $x_{t+1}$ as follows: first concatenate $n_{t+1}/n_t=p_{t+1}$ copies of $x_t$, creating $x_{t+1}'$. If the resulting word has the symbol "?" at the position $i$ for 
\begin{equation}\label{batch1}
i\in [0, n_t-1]\cup [n_{t+1}-n_t, n_{t+1}-1]\cup \left([0, n_{t+1}-1]\setminus \widetilde{R^k_{n_{t+1}}}\right),
\end{equation}
we change it to a 0 or 1 arbitrarily, creating $x_{t+1}''$. We then fill all the positions $i^k$ for $i\in [0,n_{t+1}^{1/k}]$ which still contain "?" with 0s for even $t$ and 1s for odd $t$, creating $x_{t+1}$.
\medskip

The pair $(n_t), (x_t)$ defined this way will be viable, since the first and last $n_{t-1}$ symbols of $x_t$ are not "?". Also, we have 
\[
?_t\leqslant |\widetilde{R^k_{n_t}}|=\prod_{i=1}^{t}|\widetilde{R^k_{p_i}}|=\prod_{i=1}^{t}\frac{p_i-1}{k'}=\frac{\varphi(n_t)}{(k')^t}.
\]
But $\rho_l(n_t)=(l')^t=o\left((k')^t\right)$ and $\varphi(n_t)<n_t$, so the conditions of Theorem \ref{polyequi} are satisfied, and the limit (\ref{eq:SNT2}) exists. Also, $X_x$ is regular since $?_t=o(n_t)$.

We now estimate $?_t$ from below. Notice that by the Chinese remainder theorem for any $a\in \widetilde{R^k_{n_t}}$ there are exactly $(p-1)/k'$ elements of $\widetilde{R^k_{n_{t+1}}}$ congruent to $a$ modulo $n_t$. Hence, in $x_{t+1}''$ there are at least
\[
\frac{p_{t+1}-1}{k'}?_t-2n_t 
\]
question marks. Taking into account the last step of the construction we get
\[
?_{t+1}\geqslant \frac{p_{t+1}-1}{k'}?_t-2n_t -{n_{t+1}}^{1/k},
\]
so that
\[
?_{t+1}\cdot \frac{(k')^{t+1}}{\varphi(n_{t+1})}\geqslant ?_t\cdot \frac{(k')^{t}}{\varphi(n_{t})}-\frac{2k'n_t}{p_{t+1}-1}-\frac{(n_tp_{t+1})^{1/k}}{p_{t+1}-1}\geqslant ?_t\cdot \frac{(k')^{t}}{\varphi(n_{t})}-\frac{8k'n_t}{p_{t+1}^{(k-1)/k}}.
\]
By induction and from (\ref{sumcond}) we obtain 
\begin{equation}\label{?count}
?_t\geqslant \frac{9}{10}\frac{\varphi(n_t)}{(k')^t}.
\end{equation}
We now estimate the number of $k$-th powers whose positions contain question marks in $x_{t+1}'$. 

Let $C_t=\left\lfloor{n_{t+1}}^{1/k}\right\rfloor=an_t+r$ for $a\in\mathbb{Z}$ and $r\in [0, n_t-1]$, where $a\geqslant 30$ by (\ref{powcond}). Then
\[
|\{i<C_t: x_{t+1}'\left(i^k\right)=?\}|\geqslant a|\{i<n_t: x_{t+1}'\left(i^k\right)=?\}|
\]
But we know that $\rho_k(n_t; n_t, j)=(k')^t$ for any $j\in \widetilde{R_{n_t}^k}$, and all the question marks in $n_t$ are at positions coprime with $n_t$, so the above is at least
\[
a\cdot ?_t\cdot (k')^t\geqslant  \frac{9}{10}\frac{C_t}{n_t}\cdot \frac{9}{10}\varphi(n_t)\geqslant \frac{7}{10} C_t
\]
by (\ref{?count}) and the condition (\ref{phicond}).
By construction we have
\[
|\{i<C_t: x_{t+1}'(i^k)\ne ?\text{ and }x_{t+1}''(i^k)=?\}|\leqslant 2n_t+|\{i<C_t: p_{t+1}|i\}|\leqslant 2n_t+\frac{C_t}{p_{t+1}}+1\leqslant \frac{1}{10}C_t
\]
for large $t$. Therefore in the last step of the construction of $x_{t+1}$, we fill in at least $6C_t/10$ question marks, and so
\[
\frac{1}{C_t}\sum\limits_{m<C_t}G(\sigma^{m^k}x)=\frac{1}{C_t}\sum\limits_{m<C_t}(-1)^{x_{t+1}(m^k)}\geqslant \frac{1}{C_t}\left(\frac{6}{10}C_t-\frac{4}{10}C_t\right)\geqslant \frac{1}{5}
\]
for even $t$, and 
\[
\frac{1}{C_t}\sum\limits_{m<C_t}G(\sigma^{m^k}x)\leqslant -\frac{1}{5}
\]
for odd $t$, and so the limit (\ref{div}) does not exist.
\\\\

Now let us assume that the last sentence in the statement of Theorem \ref{polyequi} is true, so that $l'$ is odd and coprime with all $p-1$ for $p|n_t$ prime. Then the polynomial $m^l$ is a permutation modulo each $n_t$, so the orbit $(\sigma^{m^l}y)$ is equidistributed in $X_x$ by Theorem \ref{permequi}.
\end{proof}

\section{Main result for $k|l$}\label{main2}
We will now discuss the remaining case $k|l$. In this case we will use a more delicate condition to guarantee equidistribution along $l$-th powers. We present a more general result concerning arbitrary polynomials in $\Z[X]$.
\begin{Lemma}\label{convcond}
Let $(x_t), (n_t)$ be a viable pair and $P\in\Z[X]$. Let $X_t\in\{0, 1, ?\}^{\mathbb{Z}}$ be the unique infinite $n_t$-periodic sequence which agrees with $x_t$ on $[0, n_t-1]$. Assume, that we have
\begin{equation}\label{shiftpow}
\sup_{a\in\mathbb{Z}}\abs{\{i\in [0, n_t-1]: X_t(P(i)+a)=?\}}=o(n_t).
\end{equation}
Then for any continuous map $F:X_x\to \mathbb{C}$ and any $y\in X_x$ the limit
\begin{equation}\label{limexists}
\lim_{N\to\infty}\frac{1}{N}\sum\limits_{m<N}F(\sigma^{P(m)}y)
\end{equation}
exists.
\end{Lemma}

\begin{proof}
We follow the proof of Theorem 6.8 in \cite{Toeplitz}. 

It suffices to show that for every continuous
$F:X_x\to\mathbb{C}$ and every $\varepsilon>0$ there exists $N_\varepsilon$ so that for every $N, M>N_\varepsilon$ and every $r\in\mathbb{Z}$, we have
\begin{equation}\label{ineq}
\abs{\frac{1}{N}\sum\limits_{m<N}F(\sigma^{P(m)+r}x)-\frac{1}{M}\sum\limits_{m<M}F(\sigma^{P(m)+r}x)}<\varepsilon
\end{equation}
Note that the above is stronger than what is needed as it shows that the convergence
in (\ref{limexists}) is uniform in $y\in X_x$. 

It is enough to consider only $F$ depending on finitely many coordinates, since the set of such functions in dense in $C(X_x)$ with the supremum norm topology. So let us assume that whenever $y, y'\in X_x$ agree on $[-C, C]$ for some $C\in\Z^+$, we have $F(y)=F(y')$. We can also assume that $F:X_x\to [0, 1]$. 

Fix $\varepsilon>0$ and $t$ such that the left hand side of (\ref{shiftpow}) is at most $\varepsilon n_t/(2C+1)$. Then fix $r\in\Z$ and let
\[
A_a=\{i\in [0, n_t-1]: X_t(P(i)+a+r)=?\},
\]
and consider the set 
\[
A=\bigcup_{a=-C}^{C}A_a.
\]
Clearly $|A|\leqslant \varepsilon n_t$, and if $i\notin A$, then 
\[
F(\sigma^{P(i)+r}x)=F(\sigma^{P(i)+n_tj+r}x)
\]
for any $j\in\mathbb{Z}$. Therefore
\[
\abs{\sum\limits_{m<kn_t}F(\sigma^{P(m)+r}x)-k\sum\limits_{m<n_t}F(\sigma^{P(m)+r}x)}\leqslant 2k|A|\leqslant 2k\varepsilon n_t
\]
for any $k\in\mathbb{Z}^+$.

Then for any $N>(\varepsilon^{-1}+1)n_t$ we can write $N=kn_t+b$ for $k>\varepsilon^{-1}$ and $b\in [0, n_t-1]$, and we have
\begin{align*}
\abs{\frac{1}{N}\sum\limits_{m<N}F(\sigma^{P(m)+r}x)-\frac{1}{n_t}\sum\limits_{m<n_t}F(\sigma^{P(m)+r}x)}\leqslant \frac{b}{N}+\left(\frac{1}{N}-\frac{1}{kn_t}\right)\cdot N+\\+\abs{\frac{1}{kn_t}\sum\limits_{m<N}F(\sigma^{P(m)+r}x)-\frac{1}{n_t}\sum\limits_{m<n_t}F(\sigma^{P(m)+r}x)}\leqslant 2\varepsilon+2\varepsilon=4\varepsilon.
\end{align*}
Therefore for $M, N>(\varepsilon^{-1}+1)n_t$ the estimate (\ref{ineq}) holds (with the constant $8\varepsilon$ in place of $\varepsilon$).

\end{proof}

We will also require the following number-theoretic result.
\begin{Lemma}\label{NTlem}
Let $p$ be prime and $k, l<p$ be positive integers. For any nonzero $a\in\mathbb{F}_p$ we have
\[
\Big\lvert\abs{\{(x, y)\in\mathbb{F}_p\times\mathbb{F}_p: x^k-y^l=a\}}-p\Big\rvert\leqslant kl\sqrt{p}.
\]
\end{Lemma}

\begin{proof}
This is an application of the Weil bound for character sums:
\begin{Th}
Fix a prime $p$, a monic polynomial $g\in \mathbb{F}_p[X]$, an integer $d|p-1$, and a multiplicative character $\chi$ of $\mathbb{F}_p$ of order $d$. Then, provided there is no polynomial $h\in\overline{\mathbb{F}_p}[X]$ such that $g=h^d$, we have
\[
\abs{\sum\limits_{x\in \mathbb{F}_p}\chi(g(x))}\leqslant (\deg g-1)\sqrt{p}.
\]
\end{Th}
For a discussion and proof see e.g. Theorem 3.1 in \cite{kowalski}.

We will apply this result to all nontrivial characters of $\mathbb{F}_p$ of order dividing $k$, and to the polynomial $g(x)=x^l+a$. Notice that $g'(x)=lx^{l-1}$, and $a\ne 0$ and $p\nmid l$ since $p>l$, so $g$ and $g'$ are coprime. Therefore $g$ is not a power of any polynomial.

For any $i\in\mathbb{F}_p$ we have
\[
|\{x\in \mathbb{F}_p: x^k=i\}|=1+\sum\limits_{\substack{\chi\ne 1 \\ \chi^k=1}}\chi(i)
\]
Therefore the number of pairs $(x, y)$ satisfying $x^k=y^l+a$ is
\[
\sum\limits_{y\in\mathbb{F}_p}|\{x\in\mathbb{F}_p: x^k=g(y)\}|=\sum\limits_{y\in\mathbb{F}_p}\left(1+\sum\limits_{\substack{\chi\ne 1 \\ \chi^k=1}}\chi(g(y))\right).
\]
This sum consists of the main term $p$ and at most $k-1$ character sums, each bounded by $(l-1)\sqrt{p}$ by the Weil bound. 

\end{proof}

In the proof of Theorem \ref{polyequi}, in (\ref{batch1}) we make sure that the only question marks in $x_t$ are at positions $\widetilde{R^k_{n_t}}$. In the case when $k|l$ this set is too large to guarantee convergence of the averages along $l$-th powers. In the following lemma we describe a set we will use instead.

\begin{Lemma}\label{A_t}
Fix distinct positive integers $k|l$, where $k>1$, and let $n$ be squarefree such that all prime divisors $p$ of $n$ satisfy $l|p-1$ and $p>(12kl)^2$. Let $A\subset (\mathbb{Z}/n\mathbb{Z})^*$ be defined as
\[
A=\{a\in((\mathbb{Z}/n\mathbb{Z})^*)^k: \abs{\{p|n: a\pmod{p}\notin (\mathbb{F}_p)^l\}}>\ln\omega(n)\}.
\]
Then 
\begin{equation}\label{est1}
|A|>\frac{\varphi(n)}{k^{\omega(n)}}\cdot\left (1-2^{-\omega(n)/4}\right)
\end{equation}
and
\begin{equation}\label{est2}
\max_{i\in\Z/n\Z}|\{x\in \Z/n\Z: x^l\in A+i|<n\cdot (2/3)^{\ln\omega(n)}.
\end{equation}
\end{Lemma}

\begin{proof}
Let 
\[
A_p=\{a\in((\Z/n\Z)^*)^k: a\pmod{p}\in (\mathbb{F}_p)^l\}. 
\]
Since $n$ is squarefree and $k|p-1$, we have $|((\Z/n\Z)^*)^k|=\varphi(n)/k^{\omega(n)}$. Since $l|p-1$, and by the Chinese remainder theorem,
\[
\frac{|A_p|}{\varphi(n)/k^{\omega(n)}}=\frac{(p-1)/l}{(p-1)/k}\leqslant \frac{1}{2}.\]
Let $\mathcal{P}$ be the family of all subsets of $\{p: p|n_t\}$ of size at least $\omega(n)-\ln\omega(n)$. From the above estimate and by the Chinese remainder theorem, for any $P\in\mathcal{P}$ we have
\[
\frac{k^{\omega(n)}}{\varphi(n)}\cdot \bigcap_{p\in P}A_p\leqslant \left(\frac{1}{2}\right)^{\omega(n)-\ln\omega(n)},
\]
and since $|\mathcal{P}|\leqslant \omega(n)^{\ln\omega(n)}$, we have
\begin{align*}
\frac{k^{\omega(n)}}{\varphi(n)}\cdot \bigcup_{P\in\mathcal{P}}\bigcap_{p\in P}A_p\leqslant \omega(n)^{\ln\omega(n)}\cdot\left(\frac{1}{2}\right)^{\omega(n)-\ln\omega(n)}=\\=\left(\frac{1}{2}\right)^{\omega(n)-\ln\omega(n)-\ln(2)(\ln\omega(n))^2}<\left(\frac{1}{2}\right)^{\omega(n)/4}.
\end{align*}
But by definition we have
\[
A=((\Z/n\Z)^*)^k\setminus\left(\bigcup_{P\in\mathcal{P}}\bigcap_{p\in P}A_p\right),
\]
so we obtain (\ref{est1}).

Now assume that (\ref{est2}) does not hold, say for some $i\in\mathbb{Z}/n\Z$. We first show that
\begin{equation}\label{manydiv}
|\{p|n: p\nmid i\}|>\ln\omega(n).
\end{equation}
Indeed, by our assumption the set $(A+i)\cap (\Z/n\Z)^l$ is nonempty, so some $a\in A$ satisfies $a+i\in(\Z/n\Z)^l$. Then 
\[
\{p|n: p|i\}\subset \{p|n: a\pmod{p}\in(\mathbb{F}_p)^l\},
\]
and by the definition of the set $A$ this set has less than $\omega(n)-\ln \omega(n)$ elements.

If $p|n$ and $p\nmid i$, then by Lemma \ref{NTlem} we have 
\[
|\{x\in \mathbb{F}_p: x^l\in (\mathbb{F}_p)^k+a\}|\leqslant \frac{1}{k}|\{(x, y)\in \mathbb{F}_p\times\mathbb{F}_p: x^l=y^k+a\}|+l\leqslant \frac{p}{k}+2kl\sqrt{p}\leqslant \frac{2}{3}p.
\]
Here the first inequality follows from the fact that the number of $k$-th roots modulo $p$ of $x^l-a$ is exactly 0 or $k$ unless $x^l-a=0$, which happens at most $l$ times. The last inequality follows from $p>(12kl)^2$.

Since $A\subset (\mathbb{Z}/n\Z)^k$, and by the Chinese remainder theorem, we have
\[
|\{x\in \Z/n\Z: x^l\in A+i\}|\leqslant \prod_{p|n}|\{x\in \mathbb{F}_p: x^l\in (\mathbb{F}_p)^k+i\}|\leqslant \left(\frac{2}{3}\right)^{\ln\omega(n)}\cdot n,
\]
where we bound each term in the product by $2p/3$ if $p\nmid i$ and by $p$ otherwise.
\end{proof}

We are now ready for the  main result of this section:
\begin{Th}
Let $k, l\in \N$ be such that $k| l$ and $k>1$. Then there exists a regular Toeplitz system $X_x$ such that for every continuous map $F:X_x\to\C$ and $y\in X_x$,  the limit
\begin{equation}\label{eq:SNT3}
\lim_{N\to\infty}\frac{1}{N}\sum_{m\leq N}F(\sigma^{m^l}y)
\end{equation}
exists, but the limit
\begin{equation}\label{div2}
\lim_{N\to\infty}\frac{1}{N}\sum_{m\leq N}G(\sigma^{m^k}x)
\end{equation}
does not exist for the continuous function $G(y)=(-1)^{y(0)}$.
\end{Th}
\begin{proof}
We first recursively define a sequence $(n_t)$ so that $n_t|n_{t+1}$, each $n_t$ is squarefree, and all prime divisors $p$ of $n_t$ satisfy $l|p-1$ and $p>(12kl)^2$. Additionally, we set $n_0=0$ and require 
\begin{equation}\label{indcond}
\sum\limits_{t=0}^{\infty}\left(2^{-\omega(n_{t+1})/4}+\frac{2n_t+\sqrt[k]{n_{t+1}}}{\varphi(n_{t+1})/k^{\omega(n_{t+1})}}\right)<\frac{1}{10},
\end{equation}
\begin{equation}\label{phicond2}
\frac{\varphi(n_t)}{n_t}>\frac{9}{10}
\end{equation}
\begin{equation}\label{powcond2}
n_{t+1}>10n_t^k.
\end{equation}

Let $A_t\subset (\Z/n_t\Z)^*$ be the set given by Lemma \ref{A_t} for $n_t$. We treat $A_t$ as a subset of $[0, n_t-1]$.

We also define a sequence $x_t\in \{0, 1, ?\}^{n_t}$. We start with $x_0=?.$ Having defined $x_t$ we define $x_{t+1}$ as follows: first we concatenate $n_{t+1}/n_t$ copies of $x_t$, creating $x_{t+1}'$. We then fill all the positions $i^k$ for $i\in [0,n_{t+1}^{1/k}]$ which still contain "?" with 0s for even $t$ and 1s for odd $t$, creating $x_{t+1}''$. If the resulting word has the symbol "?" at the position $i$ for 
\begin{equation}
i\in [0, n_t-1]\cup [n_{t+1}-n_t, n_{t+1}-1]\cup \left([0, n_{t+1}-1]\setminus A_{t+1}\right),
\end{equation}
we change it to a 0 or 1 arbitrarily, creating $x_{t+1}$. 

The pair $(n_t), (x_t)$ defined this way will be viable. Let us now show that the condition (\ref{shiftpow}) is satisfied. Also, all question marks in $x_t$ are at positions in $A_t\subset \widetilde{R^k_{n_t}}$, so
\[
?_t\leqslant |\widetilde{R^k_{n_t}}|=o(n_t),
\]
and so $X_x$ is regular.

Fix $t$ and $a\in\mathbb{Z}$. Notice that
\begin{equation}\label{excset}
\{i\in [0, n_t-1]: X_t(i^l+a)=?\}\subset\{i\in [0, n_t-1]: i^l+a\in A_t+n_t\Z\},
\end{equation}
where we let $X_t\in\{0, 1, ?\}^{\mathbb{Z}}$ be the unique infinite $n_t$-periodic sequence which agrees with $x_t$ on $[0, n_t-1]$.
By (\ref{est2}) the right hand side of (\ref{excset}) has at most $n_t\cdot (2/3)^{\ln \omega(n_t)}=o(n_t)$ elements. 
Therefore, by Lemma \ref{convcond}, we obtain convergence of averages (\ref{eq:SNT3}).

\vspace{5mm}

We now estimate $?_t$ from below. By the Chinese remainder theorem
\[
\{i\in \widetilde{R^k_{n_{t+1}}}: x_{t+1}'(i)=?\}=?_t\cdot\prod_{\substack{p|n_{t+1}\\ p\nmid n_t}}\frac{p-1}{k}=?_t\cdot \frac{\varphi(n_{t+1}/n_t)}{k^{\omega(n_{t+1}/n_t)}},
\]
 so from (\ref{est1}) we get
\[
\{i\in A_{t+1}: x_{t+1}'(i)=?\}\geqslant ?_t\cdot \frac{\varphi(n_{t+1}/n_t)}{k^{\omega(n_{t+1}/n_t)}}-\frac{\varphi(n_{t+1})}{k^{\omega(n_{t+1})}}\cdot 2^{-\omega(n_{t+1})/4},
\]
and
\[
?_{t+1}\geqslant ?_t\cdot \frac{\varphi(n_{t+1}/n_t)}{k^{\omega(n_{t+1}/n_t)}}-\frac{\varphi(n_{t+1})}{k^{\omega(n_{t+1})}}\cdot 2^{-\omega(n_{t+1})/4}-2n_t-\sqrt[k]{n_{t+1}}.
\]
By induction we get
\[
\frac{?_{t}}{|\widetilde{R^k_{n_{t}}}|}=?_{t}\cdot \frac{k^{\omega(n_t)}}{\varphi(n_t)}\geqslant 1-\sum\limits_{s=0}^{t-1}\left(2^{-\omega(n_{s+1})/4}+\frac{2n_s+\sqrt[k]{n_{s+1}}}{\varphi(n_{s+1})/k^{\omega(n_{s+1})}}\right)\geqslant \frac{9}{10},
\]
where we use (\ref{indcond}).

Now we estimate the number of question mark we fill in when going from $x_{t+1}'$ to $x_{t+1}''$. First recall that all question marks in $x_t$ are at positions in $\widetilde{R^k_{n_t}}$, and for all $a\in \widetilde{R^k_{n_t}}$ we have
\[
|\{i<n_t: i^k\equiv a\pmod{n_t}\}|=k^{\omega(n_t)},
\]
so 
\[
|\{i<n_t: X_t(i^k)=?\}|=?_t\cdot k^{\omega(n_t)}\geqslant \frac{9}{10}\varphi(n_t)\geqslant \frac{8}{10}n_t
\]
by (\ref{phicond2}).

Let $C_t=\left\lfloor{n_{t+1}}^{1/k}\right\rfloor=an_t+r$ for $a\in\mathbb{Z}$ and $r\in [0, n_t-1]$, where $a\geqslant 10$ by (\ref{powcond2}). Then
\[
|\{i<C_t: x_{t+1}'\left(i^k\right)=?\}|\geqslant a|\{i<n_t: x_{t+1}'\left(i^k\right)=?\}|\geqslant \frac{8}{10}an_t\geqslant \frac{7}{10}C_t.
\]
 Therefore when constructing $x_{t+1}''$ we fill in at least $7C_t/10$ question marks, and so
\[
\frac{1}{C_t}\sum\limits_{i<C_t}G(\sigma^{i^k}x)=\frac{1}{C_t}\sum\limits_{i<C_t}(-1)^{x_{t+1}(i^k)}\geqslant \frac{1}{C_t}\left(\frac{7}{10}C_t-\frac{3}{10}C_t\right)\geqslant \frac{2}{5}
\]
for even $t$, and 
\[
\frac{1}{C_t}\sum\limits_{i<C_t}G(\sigma^{i^k}x)\leqslant -\frac{2}{5}
\]
for odd $t$, and so the limit (\ref{div}) does not exist.
\end{proof}

\section{Density of sparse orbits}
In this section we investigate the cases, when a sparse orbit $(T^{a_n}y)$ along some sequence $(a_n)$ is dense in a Toeplitz system $X_x$ containing $y$.

We first define
\begin{Def}
Let $I\subset\Z$, and let $x\in \{0, 1\}^{I}$ and $s\in\N$. For $\varepsilon\in \{0, 1\}$ by $\Per^{(\varepsilon)}_s(x)$ we denote
\[
\{n\in\Z: (\forall k\in\N) x(n+ks)=\varepsilon\text{ whenever }n+ks\in I\}.
\]
Clearly $\Per^{(\varepsilon)}_{s}(x)$ is a disjoint union of bi-infinite arithmetic sequences with difference $s$. Therefore it is fully determined by its image under the quotient homomorphism $\pi_s: \Z\to\Z/s\Z$. Let us denote
\[
\per^{(\varepsilon)}_s(x)=\abs{\pi_s\left(\Per^{(\varepsilon)}_s(x)\right)}\in [0, s].
\]
Notice that
\[
\Per_s(x)=\Per^{(0)}_s(x)\cup \Per^{(1)}_s(x),
\]
and if $s'$ divides $s$, then $\Per_{s'}(x)\subset \Per_s(x)$.

By an essential period of a Toeplitz word $x\in\{0, 1\}^{\Z}$ we mean such an $s$ that $\Per_s(x)$ is nonempty and does not coincide
with $\Per_{s'}(x)$ for any $s'<s$.

It is easily checked that if $x'\in X_x$ is also Toeplitz, then $x$ and $x'$ have the same essential periods, and so it makes sense to talk about an essential period of a Toeplitz system $X_x$.
\end{Def}

\begin{Lemma}
Let $x$ be a Toeplitz word, and $s$ any essential period of $x$. Then for any $y\in X_x$ we have
\[
y\notin \overline{\{\sigma^ny: n\in \mathbb{Z}\text{ and } s\nmid n\}}.
\]
\end{Lemma}
\begin{proof}
Assume without loss of generality that $0\in \Per_s(x)$ but $0\notin \Per_{s'}(x)$ for any $s'<s$, and assume that $x(0)=0$. 

First notice, that for some $N$ we have 
\[
\Per^{(0)}_s(x|_{[0, N]})=\Per^{(0)}_s(x).
\]
Indeed, for any $i\in [0, s-1]\setminus \Per^{(0)}_s(x)$ there exists a natural number $N_i$ such that $x(N_i)=1$ and $N_i\equiv i\pmod{s}$. It is enough to take
\[
N=\max\{N_i: i\in [0, s-1]\setminus \Per^{(0)}_s(x)\}.
\]

By minimality of $X_x$, we can find a natural number $M$ such that any word in $\mathcal{L}_{X_x}$ of length $M$ contains $x|_{[0, N]}$ as a subword. We denote $x|_{[0, N]}=\nu$.

We will now show, that 
\[
d(y, \sigma^n y)\geqslant 2^{-M}
\]
whenever $s\nmid n$. This will clearly end the proof. 

Pick $n$ such that the above inequality does not hold. Then $y|_{[0, M]}=y|_{[n, n+M]}$; let us denote this word by $\omega$. We know that $\omega|_{[k, k+N]}=\nu$ for some $k$, so
\[
\per^{(0)}_s(x)\leqslant \per^{(0)}_s(\omega)\leqslant \per^{(0)}_s(\nu)=\per^{(0)}_s(x).
\]
Here the first inequality follows from $\omega\in\mathcal{L}_{X_x}$, the second from $\nu\subset \omega$, and the equality follows from the definition of $\nu$. In particular, both of the inequalities are actually equalities.

Since $\omega$ appears in $y|_{[0, n+M]}$ at positions starting with $0$ and $n$, we have
\[
\Per^{(0)}_s(y|_{[0, n+M]})\subset \Per^{(0)}_s(\omega)\cap \left(n+\Per^{(0)}_s(\omega)\right),
\]
but 
\[
\per^{(0)}_s(y|_{[0, n+M]})\geqslant \per^{(0)}_s(x)=\per^{(0)}_s(\omega)
\]
since $y|_{[0, n+M]}\subset x$. Hence
\[
\abs{\pi_s\left(\Per_s^{(0)}(\omega)\right)}=\per^{(0)}_s(\omega)\leqslant \abs{\pi_s\left(\Per_s^{(0)}(\omega)\right)\cap \pi_s\left(n+\Per^{(0)}_s(\omega)\right)}.
\]
The inequality must actually be an equality, and in particular we obtain
\[
\pi_s\left(\Per_s^{(0)}(\omega)\right)=\pi_s\left(n+\Per^{(0)}_s(\omega)\right).
\]
We have $\nu\subset\omega$, but $\per^{(0)}_s(\nu)=\per^{(0)}_s(\omega)$, so
\[
\pi_s\left(\Per_s^{(0)}(\nu)\right)=\pi_s\left(\Per_s^{(0)}(\omega)-k\right)=\pi_s\left(n-k+\Per^{(0)}_s(\omega)\right)=\pi_s\left(n+\Per^{(0)}_s(\nu)\right),
\]
and so
\[
\Per_s^{(0)}(\nu)=n+\Per_s^{(0)}(\nu),
\]
and by the definition of $\nu$ we get
\[
\Per_s^{(0)}(x)=n+\Per_s^{(0)}(x).
\]
But then $x(an+bs)=x(0)=0$ for all $a, b\in\Z$, and so $0\in\Per_{\gcd(n, s)}(x)$. We assumed that $s$ is an essential period of 0, so we obtain $\gcd(n, s)=s$, so $s|n$.
\end{proof}

\begin{Th}\label{resfull}
Let $X_x$ be a Toeplitz system with an essential period $s$, and pick $y\in X_x$. If $\left(\sigma^{a_n}y\right)$ is dense in $X_x$ for some sequence $(a_n)$ of integers, then $(a_n)$ gives a full set of residues modulo $s$.
\end{Th}

\begin{proof}
Pick $i\in [0, s-1]$. Since $\sigma^i y$ is in the closure of the orbit $\left(\sigma^{a_n}y\right)$ and $\sigma$ is a homeomorphism, we have
\[
y\in \overline{\{\sigma^{a_n-i}y: n\in\Z^+\}},
\]
so by the previous lemma we have $s|a_n-i$ for some $n$.
\end{proof}

In particular, we obtain

\begin{Cor}\label{fullset}
Let $(a_k)$ be a sequence of positive integers, $X_x$ be an infinite Toeplitz system, and $y\in X_x$. Assume that $(\sigma^{a_k}y)$ is dense in $X_x$. Then there exists a strictly increasing sequence $(n_t)$ of positive integers such that $n_t|n_{t+1}$ and $(a_k)$ gives a full set of residues modulo each $n_t$.
\end{Cor}

This easily implies, that prime orbits are never dense in infinite Toeplitz systems. In fact, we can say this even about $l$-almost primes:
\begin{Def}
For $l\in\mathbb{Z}^+$ we denote
\[
\mathbb{P}_l=\{n\in\mathbb{Z}^+: 1<\Omega(n)\leqslant l\},
\]
and we call elements of this set $l$-almost primes.
\end{Def}

\begin{Th}
Let $X_x$ be an infinite Toeplitz system and $y\in X_x$, and fix $l\in\Z^+$. The orbit $\{\sigma^{n}y: n\in\mathbb{P}_l\}$ is not dense in $X_x$.
\end{Th}

\begin{proof}
Assume otherwise, and let $(n_t)$ be the sequence given by Corollary (\ref{fullset}). Then $\Omega(n_{l+2})>l$, so $n_{l+2}$ does not divide any element of $\mathbb{P}_l$, and so we reach a contradiction.
\end{proof}

We now discuss the case of polynomial orbits. In this case Theorem \ref{resfull} has a converse:
\begin{Th}\label{dense}
Let $X_x$ be a Toeplitz system and $P\in\Z[X]$ be a permutation modulo each essential period of $X_x$. Then $(\sigma^{P(n)}y)$ is dense in $X_x$ for any $y\in X_x$.
\end{Th}
\begin{proof}
It is enough to show that this orbit visits each set of the form \[U=\{z\in X_x: z[-k, k]=\omega\}\] for $\omega\in \mathcal{L}_{X_x}$ of length $2k+1$. Fix such an $\omega$, and let $x[a, a+2k]=\omega$. Then for $i\in [a, a+2k]$ we let $p_i$ be the essential period of $x_i$ in $x$. This way the word $\omega$ appears in $x$ starting at positions $a+np$ for $n\in\Z$ and $p=\lcm\{p_i: i\in [a, a+2k]\}$. Since the polynomial $P$ is a permutation modulo each $p_i$, it is also a permutation modulo $p$.

Now, since $y\in X_x$, the word $\omega$ must appear in $y$ periodically with period $p$, say starting at positions $a'+np$ for $k\in\Z$. If $b>0$ is such that $P(b)\equiv a'+k\pmod{p}$, then $y[P(b)-k, P(b)+k]=\omega$, so $\sigma^{P(b)}y\in U$.
\end{proof}
\begin{Th}
Let $P$ be a polynomial which is not permutative. Let $X_x$ be an infinite Toeplitz system, and $y\in X_x$. The orbit $\{\sigma^{P(n)}y: n\in\Z^+\}$ is not dense in $X_x$.
\end{Th}
\begin{proof}
Follows directly from the definition of permutativity and Corollary \ref{fullset}.
\end{proof}

\begin{Th}\label{permequi}
Let $P$ be a permutative polynomial, and let $x$ be a Toeplitz word such that $X_x$ is regular and $P$ is a permutation modulo each essential period of $X_x$. Then for any $y\in X_x$ and any continuous $F:X_x\to X_x$ we have
\[
\lim_{N\to\infty}\frac{1}{N}\sum\limits_{m<N}F\left(\sigma^{P(m)}y\right)=\int_{X_x}F d\mu,
\]
uniformly in $y$, where $\mu$ is the unique $\sigma$-invariant measure on $X_x$.
\end{Th}
\begin{proof}
Let $(x_t), (n_t)$ be a viable pair producing the Toeplitz word $x$, and such that each $n_t$ is a least common multiple of some essential periods of $X_x$. This way $P$ is a permutation modulo each $n_t$. We proceed as in the proof of Lemma \ref{convcond}. It is enough to show that
\[
\lim_{N\to\infty}\frac{1}{N}\sum\limits_{m<N}F\left(\sigma^{P(m)+r}x\right)=\int_{X_x}F d\mu
\]
uniformly in $r\in\Z$. We again assume that $F$ only depends on coordinates from $-C$ to $C$ for some $C\in\Z^+$, and that $F:X_x\to[0, 1]$. Fix $t$, and let
\[
A_a=\{i\in [0, n_t-1]: X_t(i+a+r)=?\}
\]
and 
\[
A=\bigcup_{a=-C}^{C}A_a.
\]
Then $|A|\leqslant (2C+1)?_t$, and for $i\notin A$ we have
\[
F(\sigma^{P(i)+a+r}x)=F(\sigma^{P(i)+a+r\pmod{n_t}}x),
\]
so since $P+r$ is a permutation modulo $n_t$, we get
\[
\abs{\sum\limits_{m<n_t}F(\sigma^{P(m)+r}x)-\sum\limits_{m<n_t}F(\sigma^{m}x)}\leqslant 2\cdot (2C+1)?_t
\]
and as in the proof of Lemma \ref{convcond} we obtain
\[
\abs{\frac{1}{N}\sum\limits_{m<N}F(\sigma^{P(m)+r}x)-\frac{1}{n_t}\sum\limits_{m<n_t}F(\sigma^{m}x)}\leqslant \frac{2\cdot (2C+1)?_t}{n_t}+\frac{n_t}{N}+\frac{n_t}{N-n_t}
\]
for $N>n_t$. Since
\[
\lim\limits_{t\to\infty}\frac{1}{n_t}\sum\limits_{m<n_t}F(\sigma^{m}x)=\int_{X_x}F d\mu,
\]
we obtain the desired convergence, and it is clearly uniform in $r\in\Z$.
\end{proof}

In particular, using Theorem \ref{resfull} we obtain Theorem \ref{denseequi}.

This theorem is no longer true if we remove the assumption of regularity -- the method we used in Section \ref{main1} would give the relevant counterexamples. We will omit the details, and instead adapt (Example 2.3, \cite{iwanik}) to show that there are in fact strictly ergodic counterexamples (which we could not ensure with methods from Section \ref{main1}). 
\begin{Th}\label{iwacounter}
Let $P\in\Z[X]$ be permutative and let $\deg P>1$. There exists a strictly ergodic Toeplitz system $X_x$, for which each orbit $(\sigma^{P(n)}y)$ for $y\in X_x$ is dense in $X_x$, but the limit
\begin{equation}\label{div3}
\lim_{N\to\infty}\frac{1}{N}\sum_{m\leq N}G(\sigma^{P(m)}x)
\end{equation}
does not exist for the continuous function $G(y)=(-1)^{y(0)}$.
\end{Th}

\begin{proof}
Let $M$ be the absolute value of the leading coefficient of $P$. Since $d>1$, by replacing $P(x)$ with $\pm P(x+k)$ for some large $k$ we can assume that $P(n+1)-P(n)>n$ and $P(n)>Mn^d$ for $n\geqslant 0$. Let $(n_t)$ be a strictly increasing sequence starting with $n_0=1$ satisfying $n_t|n_{t+1}$ such that $P$ is permutative modulo each $n_t$. By restricting to a subsequence, we can additionally assume that $n_{t+1}>(M+1)(10n_t)^{d}$ and
\begin{equation}\label{recipsum}
\sum\limits_{t=0}^{\infty}\frac{2n_t}{n_{t+1}}<1/5,
\end{equation}
and that all $m_t:=n_{t+1}/n_t$ are of the same parity.
We now recursively define sequences of words $B_t^{(0)}, B_t^{(1)}\in\{0, 1\}^{n_t}$. We start with $B_0=0$ and $B_0'=1$.

Given $B_t^{(0)}$ and $B_t^{(1)}$, we let 
\[
A_t=\{i\in [0, n_t-1]: B_t^{(0)}(i)\ne B_t^{(1)}(i)\}.
\]
We let
\[
B_{t+1}^{(0)}=B_t^{(\varepsilon_t(0))}B_t^{(\varepsilon_t(1))}\ldots B_t^{(\varepsilon_t(m_t-1))}
\]
for a sequence $\varepsilon_t\in \{0, 1\}^{m_t}$ satisfying the following:
\begin{enumerate}
\item $\varepsilon_t(0)=0$ and $\varepsilon_t(m_t-1)=1$.
\item for $i>n_t$ satisfying $P(i)<n_{t+1}-n_t$ and $P(i)\in A_t+n_t\Z$ we have $B_{t+1}^{(0)}(i)=0$ if $t$ is odd and 1 if $t$ is even
\item \[|\{i\in [0, m_t-1]:\varepsilon_t(i)=0\}|=\begin{cases}m_t/2\quad&\text{if $m_t$ is even}\\ (m_t+1)/2\quad &\text{if $m_t$ is odd}\end{cases}\]
\end{enumerate}
Notice that the second condition can be satisfied, since for $i$ in the specified range we have $P(i+1)-P(i)>n_t$, and so the values of $\lfloor P(i)/n_t\rfloor\in [1, m_t-2]$ are all distinct, and so we can pick each of the $\varepsilon_t(\lfloor P(i)/n_t\rfloor)$ separately and without interfering with condition (1). This way we specify the value of $\varepsilon_t$ at no more than 
\[
2+\sqrt[d]{n_{t+1}}<3\frac{n_{t+1}}{\sqrt[d]{n_{t+1}}}<3\frac{n_{t+1}}{10n_t}<m_t/2
\]
places, so we can now fulfill condition (3).

We then define
\[
B_{t+1}^{(1)}=B_t^{(\varepsilon_t'(0))}B_t^{(\varepsilon_t'(1))}\ldots B_t^{(\varepsilon_t'(m_t-1))},
\]
where $\varepsilon'(i)=1-\varepsilon(i)$ for $i\notin\{ 0, m_t-1\}$, but $\varepsilon'(0)=0$ and $\varepsilon'(m_t-1)=1$.

Letting $W_t=\{B_t^{(0)}, B_{t}^{(1)}\}$ we construct $x\in \{0, 1\}^{\Z}$ as in section \ref{toep}; by Fact \ref{iwatoep} the word $x$ is then Toeplitz.
 If the $m_t$ are even, by induction we have
\[
\ap(B_t^{(\varepsilon)}, B_s^{(\varepsilon')})=\frac{1}{2}
\]
for any $t<s$ and $\varepsilon, \varepsilon'\in\{0, 1\}$. By Theorem \ref{unierg} in this case $X_x$ is uniquely ergodic.

If the $m_t$ are odd, by a similar computation as in Example 2.3 in \cite{iwanik} we obtain
\[
\ap(B^{(\varepsilon)}_t, B^{(\varepsilon')}_s)=\frac{1}{2}\left(1\pm \frac{1}{m_tm_{t+1}\cdots m_{s-1}}\right),
\]
where the sign is positive if $\varepsilon=\varepsilon'$ and negative otherwise. Hence in this case for any $t, \varepsilon, \varepsilon'$ we have
\[
\lim\limits_{s\to\infty}\ap(B_t^{(\varepsilon)}, B_s^{(\varepsilon')})=\frac{1}{2},
\]
so by Theorem \ref{unierg} the system $X_x$ is uniquely ergodic.

All essential periods of $X_x$ are clearly divisors of some $n_t$. By Theorem \ref{dense} we therefore see that $(\sigma^{P(n)}y)$ is dense in $X_x$ for any $y\in X_x$.

All that remains is to see that the limit (\ref{div3}) does not exist. 

We first estimate $|A_t|$ inductively: we have $|A_0|=1$, and
\[
|A_{t+1}|\geqslant (m_t-2)|A_{t}|,
\]
so
\[
\frac{1}{n_t}|A_t|\geqslant \prod_{s=0}^{t-1}\left(1-\frac{2n_t}{n_{t+1}}\right)\geqslant \frac{4}{5}
\]
by (\ref{recipsum}). 

Let us now estimate the number of $i> n_t$ that satisfy the assumption of condition (2). Since $P$ is permutative, among any $n_t$ consecutive values of $i$ precisely $|A_t|$ satisfy $P(i)\in A_t+n_t\Z$. Let $C_t=\max\{i\in\Z^+: P(i)<n_{t+1}-n_t\}$. For large $t$ we have $C_t>\sqrt[d]{n_{t+1}/(M+1)}>10n_t$, so
\[
\abs{\{i>n_t: P(i)<n_{t+1}-n_t\text{ and }P(i)\in A_t+n_t\Z\}}>|A_t|\cdot \frac{(C_t-2n_t)}{n_t}>\frac{4}{5}\cdot \frac{4}{5}C_t>\frac{3}{5}C_t.
\]
Therefore for large even $t$ we have
\[
\frac{1}{C_t}\sum\limits_{m<C_t}G(\sigma^{P(m)}x)=\frac{1}{C_t}\sum\limits_{m<C_t}(-1)^{x_{t+1}(P(m))}\geqslant \frac{1}{C_t}\left(\frac{3}{5}C_t-\frac{2}{5}C_t\right)\geqslant \frac{1}{5}
\]
and for large odd $t$ we have
\[
\frac{1}{C_t}\sum\limits_{m<C_t}G(\sigma^{P(m)}x)\leqslant-\frac{1}{5},
\]
and so the limit (\ref{div3}) does not exist.
\end{proof}

\section{Acknowledgments}
I would like to thank Adam Kanigowski for introducing me to the study of sparse orbits in Toeplitz systems, and for helpful discussions and comments regarding this note.

\end{document}